\documentclass[a4paper,11pt]{article}
\usepackage{amsmath}
\usepackage{amssymb}
\usepackage{amsthm}
\newtheorem{theorem}{Theorem}[section]
\newtheorem{corollary}[theorem]{Corollary}
\newtheorem{lemma}[theorem]{Lemma}
\newtheorem{proposition}[theorem]{Proposition}

\newtheorem{conjecture}[theorem]{Conjecture}
\newtheorem{definition}[theorem]{Definition}
\theoremstyle{remark}

\newcommand{\N}{{\mathbb{N}}}
\newcommand{\Z}{{\mathbb{Z}}}
\newcommand{\F}{{\mathbb{F}}}
\newcommand{\R}{{\mathbb{R}}}
\begin{document}

\author{Christian Elsholtz, Adam J.~Harper\thanks{AJH was supported by a Doctoral Prize from the Engineering and Physical Sciences Research Council of the United Kingdom, and by a postdoctoral fellowship from the Centre de recherches math\'{e}matiques in Montr\'{e}al. He would also like to thank the Technische Universit\"at Graz for hospitality during his visit in September 2011, when the research for this paper started.}}
\title{Additive decompositions of sets\\ with restricted prime factors}

\maketitle
\begin{abstract}
We investigate sumset decompositions of quite general sets with restricted prime factors. We manage to handle certain sets, such as the smooth numbers, even though they have little sieve amenability, and conclude that these sets cannot be written as a ternary sumset. This proves a conjecture by S\'{a}rk\"ozy. We also clean up and sharpen existing results on sumset decompositions of the prime numbers.
\end{abstract}

\section{Introduction}{\label{sec:introduction}}

At the beginning of his two volume treatise on additive number theory \cite{Ostmann:1956} Ostmann introduces the concepts of additive decomposability, and asymptotic additive decomposability, of a set:

\begin{definition}[See \cite{Ostmann:1956}, 
vol.~1, p.~5]{\label{def:asumptoticirreducible}}
Let ${\cal S}$ be a set of positive integers.
We say that ${\cal S}$ is \emph{asymptotically
 additively irreducible}
(or we say that \emph{no asymptotic additive
decomposition exists}) if
there do not exist two sets of positive integers ${\cal A}$ and ${\cal 
B}$, with at least two elements each,
 such that for a sufficiently large $x_0$:
\[({\cal A+B})\cap [x_0, \infty)= {\cal S}\cap [x_0,\infty).\]
Here the sumset 
of ${\cal A}$ and ${\cal B}$ is defined by
$\mathcal{A} + \mathcal{B} := \{a+b : a \in \mathcal{A}, b \in \mathcal{B}\}$.

If ${\cal S}$ is asymptotically additively reducible we write this as 
${\cal S}\sim{\cal A}+{\cal B}$.
\end{definition}

Wirsing \cite{Wirsing:1953} showed that almost all sets of integers
are asymptotically additively irreducible. But it seems very difficult to prove whether a given set $\mathcal{S}$ that ``occurs in nature'' is asymptotically additively irreducible or not. Shortly we will describe some of the results obtained previously. The main purpose of this paper is to develop a more structural approach, rather than only looking at special examples, and to identify a large family of multiplicatively defined sets (including sets of smooth numbers, the set of primes, and sets of integers composed from dense subsequences of the primes) for which one has the same consequences: if the sets can be asymptotically additively decomposed, then both summands must have counting functions about $x^{1/2+o(1)}$ (see Theorem \ref{genthm} for a general statement of this type); moreover there is no asymptotic additive decomposition into three sets.

Methodologically, the new ingredients of this paper are the use of Selberg's sieve, in combination with the large and larger sieves, to handle less regular sets; and the use of estimates for sums of general multiplicative functions, to make the process of combining the sieve arguments more efficient and general. The former is ultimately what allows us to treat the smooth numbers example, and the latter is what allows us to abstract the arguments into the general Theorem \ref{genthm}, and to obtain some improvements in the prime numbers example. We also borrow an idea of S\'{a}rk\"ozy \cite{Sarkozy:2012} from the finite field case, by using a sumset inequality of Ruzsa \cite{Ruzsa:2009} to deduce ternary indecomposability results from our binary counting results. This is simpler and more widely applicable than the previous local approach to such deductions.

\vspace{12pt}
The most frequently studied case in the literature is Ostmann's conjecture.
\begin{conjecture}[Ostmann, \cite{Ostmann:1956}, volume 1, page 13]
The set $\mathcal{P}$ of primes is asymptotically additively irreducible.
\end{conjecture}

This conjecture, which is sometimes called the ``inverse Goldbach conjecture'', has sparked considerable interest. Erd\H{o}s \cite{Erdos:1977} called it ``an old and very fascinating conjecture'', and in \cite{Erdos:1986} he called it ``one of the most interesting unconventional problems on primes'', and he included it in several more problem collections. A variety of methods have been applied to try to attack the problem: congruence, gap or
density considerations (Ostmann \cite{Ostmann:1956},
Laffer and Mann \cite{LafferandMann:1964}, 
Hornfeck (see Ostmann \cite{Ostmann:1956}, volume I, page 13), 
Puchta \cite{Puchta:2002}); 
the small sieve (Hornfeck \cite{Hornfeck:1954}, 
Mann \cite{Mann:1965},
Croot and Elsholtz \cite{CrootandElsholtz:2005}); 
the large sieve 
(Pomerance, S\'{a}rk\"ozy and Stewart 
\cite{PomeranceandSarkozyandStewart:1988}, 
Hofmann and Wolke \cite{HofmannandWolke:1996}, 
Elsholtz \cite{Elsholtz:2001archiv}, 
Croot and Elsholtz \cite{CrootandElsholtz:2005}); 
and a combination of the large sieve method with Gallagher's larger
sieve (Elsholtz \cite{Elsholtz:2001mathematika}, \cite{Elsholtz:2006}). 
See the first author's paper \cite{Elsholtz:2006} 
for some further historical references.

It is known that a ternary decomposition ${\cal P}\sim {\cal A}+{\cal B} +{\cal C}$ is impossible, where ${\cal A,B,C}$ are sets of positive integers with at least two elements each (see \cite{Elsholtz:2001mathematika}, \cite{Elsholtz:2002}). It is also known that if ${\cal P}\sim {\cal A}+{\cal B}$, then 
$\max ({\cal A}(x), {\cal B}(x)) \ll x^{1/2} (\log x)^2$ (see
\cite{Elsholtz:2006}) and
$\frac{x}{\log x}\ll {\cal A}(x) {\cal B}(x) \ll x$ (see 
\cite{Wirsing:1972}, \cite{PomeranceandSarkozyandStewart:1988}, and a survey \cite{Elsholtz:2006}). 
Here $\mathcal{A}(x) := \#\{n \leq x : n \in \mathcal{A}\}$
is the counting function, similarly for the counting function $\mathcal{B}(x)$.

Some other related problems have been studied. For example, it is known that the set of shifted primes ${\cal P}+c$ cannot be
asymptotically decomposed into a product set ${\cal A}{\cal B}$, see 
\cite{Elsholtz:2008}.

Other sets of integers have been investigated for decomposability but Erd\H{o}s and Nathanson \cite{ErdosandNathanson:1978} write ``very little is understood about the general problem of the decomposition of sets in the form $S=A+B$, where $|A|\geq 2$ and $|B|\geq 2$.'' For some general results on additive decomposability of integers see S\'{a}rk\"{o}zy \cite{Sarkozy:1962, Sarkozy:1964}. The set of integer squares is irreducible, simply because the gaps between squares are increasing, but Erd\H{o}s conjectured that the set of squares $\{n^2\leq N^2\}$ remains irreducible even if one is allowed to change $o(N)$ of the elements. Partial results are due to S\'ark\"{o}zy and Szemer\'edi \cite{SarkozyandSzemeredi:1965}.

We also mention that various work has been done recently on finite field
versions of the above questions. To give a few examples, S\'{a}rk\"{o}zy 
\cite{Sarkozy:2012} investigated the decomposability of the set of squares
modulo primes, and Dartyge and S\'{a}rk\"{o}zy \cite{DartygeandSarkozy:2013}
investigated the set of primitive roots modulo primes. 
Extensions of their results in various directions are due
to Shparlinski \cite{Shparlinski}, Shkredov \cite{Shkredov}, and
Gyarmati, Konyagin and S\'{a}rk\"{o}zy \cite{GyarmatiKonyaginandSarkozy}. One also has finite field analogues of Wirsing's result \cite{Wirsing:1953} mentioned above, for example Alon, Granville and Ubis \cite{AlonGranvilleandUbis:2010} proved that the number of sets ${\cal S}\subset\Z/p\Z$ that can be written as a sumset ${\cal S}={\cal A}+{\cal B}$ is small, namely $2^{p/2+o(p)}$.

\vspace{12pt}
Following the solution \cite{Elsholtz:2001mathematika}
of the ternary version of Ostmann's problem, mentioned
above, S\'{a}rk\"{o}zy \cite{Sarkozy:2001} posed the binary and ternary additive decomposability of the smooth numbers 
as an open problem.
A transfer of the methods from \cite{Elsholtz:2001mathematika} did not appear
to work.

\begin{definition}
A number is said to be {\em $y$-smooth} if all of its prime factors are at most $y$. We will write $\mathcal{S}_{y}$ for the set of all $y$-smooth numbers. More generally, if $f(n)$ is a function, then we will write $\mathcal{S}_{f(n)}$ for the set of all numbers $n$ that are $f(n)$-smooth.
\end{definition}

\begin{conjecture}[S\'{a}rk\"{o}zy \cite{Sarkozy:2001}]
Let $0 < \epsilon < 1$ be fixed. Then the set $\mathcal{S}_{n^{\epsilon}}$, of integers $n$ with no prime factor greater than $n^{\epsilon}$, is asymptotically additively irreducible.
\end{conjecture}

\section{Announcement of results}{\label{sec:announcement}}
\subsection{Sets of smooth numbers}
An argument along the lines of \cite{Elsholtz:2001mathematika} does not seem to work unless the size $\mathcal{S}(x)$ of the target set can be well estimated by a sieving procedure. This excludes some natural sets, such as sets $\mathcal{S}_{f(n)}$ of smooth numbers, from consideration. In this paper we will prove an additive irreducibility theorem for sets that need not be well controlled by the sieve in this sense, but have certain other regularity properties. We defer a general statement, which is the main result of this paper but is a bit technical, to Theorem \ref{genthm}, but as a special case we achieve the following results on smooth numbers. In particular, the ternary version of the above mentioned conjecture of S\'{a}rk\"{o}zy will be solved for all small $\epsilon$. 

\begin{theorem}\label{smooth1}
There exist a large absolute constant $D > 0$, and a small absolute constant
$\kappa > 0$, such that the following is true. Suppose $f(n)$ is an increasing
function such that $\log^{D}n \leq f(n) \leq n^{\kappa}$ for large $n$, and such that $f(2n) \leq f(n)(1+(100\log f(n))/\log n)$. Then if ${\cal
  A}+{\cal B}\sim {\cal S}_{f(n)}$, where ${\cal A}$ and ${\cal B}$ 
contain at least two elements each, we have
$$ \max(\mathcal{A}(x), \mathcal{B}(x)) \ll \sqrt{x}\log^{4}x . $$ 
\end{theorem}

\begin{corollary}\label{smooth2}
If $f(n)$ is as in Theorem \ref{smooth1},
 then a ternary decomposition ${\cal A}+{\cal B}+{\cal C}\sim {\cal S}_{f(n)}$, 
where ${\cal A,B}$ and ${\cal C}$ contain at least two
elements each, does not exist.
\end{corollary}

In particular, one can take $f(n)=n^{\epsilon}$ in Corollary \ref{smooth2} for any fixed $0 < \epsilon \leq \kappa$.

\vspace{12pt}
Let us explain briefly how we are able to handle the smooth numbers case, referring the reader to section \ref{sec:genmachinery} for fuller details. The proof 
strategy of \cite{Elsholtz:2001mathematika}, where the primes were handled, can be summarized as follows:
suppose that ${\cal A}+{\cal B}\sim {\cal S}$, where the set ${\cal S}$ shall
be investigated. Since ${\cal S}$ has restricted prime factors, a sieve argument shows that if ${\cal A}(x)$ is a certain size
then ${\cal B}(x)$ is rather less than $\frac{{\cal S}(x) }{{ \cal A}(x)}$, which is a
contradiction since we certainly have
$$ {\cal S}(x) + O(1) = ({\cal A}+{\cal B})(x) \leq \mathcal{A}(x) \mathcal{B}(x) $$
if ${\cal A}+{\cal B}\sim {\cal S}$. Studying different ranges of ${\cal A}(x)$ this eventually leads
to good lower and upper bounds on ${\cal A}(x)$ and ${\cal B}(x)$. Then one can
deduce a ternary indecomposability result, which in this paper will be done using general sumset inequalities.

If ${\cal S}$ is a set of smooth numbers,
and if one tries a direct application of the sieve methods applied to the set of primes, then 
one gets stuck right at the beginning when trying to find basic upper bounds
for $k$-tuples of smooth numbers. For with $u=\frac{\log x}{\log y}$, and $\varrho$ denoting the Dickman function, it is known that ${\cal S}_y(x) \sim \varrho(u) x \approx
\frac{x}{u^u}$ on a wide range. If $a_1, \dotsc, a_k$ are distinct integers
one would therefore expect that 
\[ \#\{n \leq x: n+a_1, \dotsc , n+a_k \in {\cal S}_y \} \approx 
\varrho(u)^k x \approx \frac{x}{u^{uk}}.\]
By a usual application of sieve methods (sieving the integers $n \leq x$ by primes $y < p \leq \sqrt{x}$, say) one would rather prove something like
\[ \#\{n \leq x: n+a_1, \dotsc , n+a_k \in {\cal S}_y\} \ll_k \frac{x}{u^{k}}.\]
If $k<u$ this estimate is even worse than the ``trivial'' bound ${\cal S}_y(x)$ for the left hand side, and so the method outlined above cannot proceed. Note that in the primes case one does obtain correct order upper bounds for prime $k$-tuples.

In contrast to this, we will use Selberg's sieve and some other ingredients (see the proof of Proposition \ref{smallkbv}, below) to show something like
\[ \#\{n \leq x: n+a_1, \dotsc , n+a_k \in {\cal S}_y \} \ll_k
 \frac{x}{u^{u+(k-1)}},\]
and will show that this is sufficient to make our proof run.

\subsection{Sets consisting of prime factors from a certain set}
As a second example we study semigroup-type sets ${\cal Q}({\cal T})$ of integers composed of certain prime factors $\mathcal{T}$ only.
\begin{theorem}
Let ${\cal T}$ be any finite set of primes, and let 
\[{\cal Q}({\cal T})=\{n \in \N: p|n 
\Rightarrow p \in {\cal T}\}.\]
Then ${\cal Q}({\cal T})$ does not have an asymptotic additive decomposition
into two sets. 
\end{theorem}
\begin{proof}
Improving results of Thue, Siegel, Mahler and making use of Baker's theory on linear forms in logarithms,
 Tijdeman \cite[page 319]{Tijdeman:1973} proved
that the sequence $(q_n)$ of those integers composed 
of finitely many given primes has very large gaps.
Here we only make use of the weaker statement that $\liminf_{n \rightarrow \infty} (q_{n+1}-q_n) =\infty$.
On the other hand, if ${\cal A}+{\cal B}\sim {\cal Q}({\cal T})$ and
$a_1,a_2 \in {\cal A}$, and ${\cal B}$ is infinite, then $a_2-a_1$ must be a 
(bounded) gap of infinitely many pairs of elements of
 ${\cal Q}({\cal T})$. This is a contradiction.
\end{proof}

From Tijdeman \cite[Theorem 7]{Tijdeman:1973} it similarly follows that there exists an infinite set ${\cal T}$ with the same conclusion.

Using the sieve-based methods of this paper, rather than gap considerations, we shall prove some related results.

\begin{theorem}{\label{thm:tauupperbound}}
Let ${\cal T}$ denote any set of primes with 
\[  \sum_{\begin{subarray}{c}
 p \leq x\\
p \in {\cal T}
\end{subarray}}
\frac{\log p}{p} =\tau \log x  + C+o(1). \]
Here $0< \tau < 1$ and $C$ denote any real constants.
Let 
\[{\cal Q}({\cal T})=\{n \in \N: p|n 
\Rightarrow p \in {\cal T}\}.\]
If ${\cal Q}({\cal T}) \sim {\cal A}+{\cal B}$, where ${\cal A},{\cal B}$
contain at least two elements each, then
\[\max ({\cal A}(x), {\cal B}(x)) \ll_{\cal{T}} x^{1/2}(\log x)^{4}.\]
\end{theorem}

This theorem includes the case where ${\cal T}$ consists of the primes in one or several arithmetic progressions. 

\begin{corollary}{\label{cor:semigroup}}
Let ${\cal T}$ satisfy the hypotheses of Theorem \ref{thm:tauupperbound}. Then the set
${\cal Q}({\cal T})=\{n \in \N: p|n 
\Rightarrow p \in {\cal T}\}$
cannot be asymptotically decomposed into three sets with at least two elements each. 
\end{corollary}

It was previously proved in \cite{Elsholtz:2006} that the product
\[ {\cal A}(x) {\cal B}(x) \ll_{\tau, C} x(\log x)^{2\tau}, \]
and that there are various examples, based on properties of sums of two squares (which is a class of examples of some conjectural relevance to Ostmann's problem, see section \ref{sec:oststatement}), where ${\cal A}+{\cal B}$ is a large subset of ${\cal Q}({\cal T})$ and
\[ {\cal A}(x) {\cal B}(x) \gg_{\tau, C} x(\log x)^{-1/2}.\]
Obtaining a bound on $\max({\cal A}(x), {\cal B}(x))$, as we do in this paper as a straightforward corollary of our general Theorem \ref{genthm}, is a stronger kind of information and allows us to deduce Corollary \ref{cor:semigroup}.

\subsection{Ostmann's problem revisited}{\label{sec:oststatement}}
\begin{theorem}\label{Ostmann}
Suppose that ${\cal P}\sim {\cal A}+{\cal B}$, where ${\cal P}$ is the set of all primes and ${\cal A},{\cal B}$
contain at least two elements each. Then the following bounds
hold:
\[ \frac{x^{1/2}}{\log x \log\log x} \ll {\cal A}(x)\ll x^{1/2} \log\log x.\]
The same bounds hold for ${\cal B}(x)$.
\end{theorem}

It was known before that ${\cal A}(x) {\cal B}(x)\ll x$ (for a survey of references see \cite{Elsholtz:2006}) and $\max ({\cal A}(x), {\cal B}(x))\ll x^{1/2}(\log x)^2$ in the infinite version of the problem\footnote{One can also consider the slightly different ``finitary'' version of the decomposition problem, where we wish to know whether or not ${\cal P}\cap[x_0,x]= {\cal A}+{\cal B}$ for sufficiently large (fixed) $x_0$ and $x$. Previously one only had the weaker bound $\max (\#{\cal A}, \#{\cal B}) \ll x^{1/2}(\log x)^4$ (see \cite{Elsholtz:2001mathematika}) in the finitary version, due to a problem with the ranges of primes for sieving, but our new bound $\ll x^{1/2} \log\log x$ applies in this case too.}. In this case, our improvement is derived by using a lower bound estimate of Hildebrand~\cite{Hildebrand:1987} for sums of multiplicative functions to make part of the sieving argument more efficient. See section \ref{sec:ostmannrevisited}, below.

It is an odd feature of Ostmann's problem that we know that modulo many primes, the putative summand sets ${\cal A}$ and ${\cal B}$ both must lie in about $\frac{p}{2}$ residue classes. See \cite{Elsholtz:2001mathematika}, or the proof of Theorem \ref{Ostmann} in section \ref{sec:ostmannrevisited}. Heuristically one would therefore expect that the sets ${\cal A}$ and ${\cal B}$ 
must have some special structure, i.e. are essentially the values of a quadratic polynomial. So ${\cal A} +{\cal B}$ must resemble somewhat a binary quadratic form, which we certainly wouldn't expect to coincide with the set of primes.
But currently this heuristic cannot be substantiated rigorously. Some further progress in this direction will appear in forthcoming
work of Green and Harper.

\section{Some background results}{\label{sec:background}}
\subsection{Lemmas from sieve theory}
The main ingredients of our proofs are Selberg's sieve, Montgomery's large sieve inequality, and Gallagher's larger sieve, combined in various ways. We state versions of these sieve results now.

\begin{lemma}[Selberg, see Theorem 7.1 of \cite{FriedlanderandIwaniec:2010}]{\label{lem:Selberg}}
Let ${\cal P}_0$ denote a finite subset of the primes, and let
${\cal{C}}$ be a finite, non-empty set of integers. 
Let $r_1, \dotsc , r_k$ be distinct integers.
Define a set $\mathcal{B} \subseteq \mathcal{C}$ in the following way:
$$ \mathcal{B} := \{c \in \mathcal{C} : c \not\equiv r_{i} \bmod p, \; 
 1 \leq i \leq k, \; p \in {\cal P}_0\}. $$
Finally, suppose that $\omega_{\cal B} : {\cal{P}}_0 \rightarrow \R$ is any function with $0 \leq \omega_{\cal B}(p) < p$, and that $Q$ is any positive integer. Then we have
$$ \#\mathcal{B} \leq \frac{\#\mathcal{C}}{L} + \sum_{\substack{d \leq Q^{2},
    \\ p \mid d \Rightarrow p \in {\cal P}_0}} \mu^{2}(d) \tau_{3}(d)
\left|\#\{c \in \mathcal{C}: d \mid (c-r_{1})(c-r_{2})\dotsm (c-r_{k})\} - \#\mathcal{C} \prod_{p \mid d} \frac{\omega_{\cal B}(p)}{p} \right|, $$
where $\mu(d)$ denotes the M\"obius function, $\tau_{3}(d) := \#\{(u,v,w) \in \N^{3}: uvw=d\}$, and
$$ L:= \sum_{\substack{q \leq Q, \\ p \mid q \Rightarrow p \in {\cal P}_0}}
\mu^2(q) \prod_{p \mid q} \dfrac{\omega_{\cal B}(p)}{p-\omega_{\cal B}(p)}. $$
\end{lemma}

Lemma \ref{lem:Selberg} follows e.g. from Theorem 7.1 of 
Friedlander and Iwaniec's book~\cite{FriedlanderandIwaniec:2010} if one chooses
$a_{n} = \sum_{c \in \mathcal{C}} \textbf{1}_{(c-r_{1})(c-r_{2})\dotsm (c-r_{k})=n}$, $P = \prod_{p \in {\cal P}_0}p$, and $g(p)=\omega_{\mathcal{B}}(p)/p$ in their theorem.

\begin{lemma}[Montgomery \cite{Montgomery:1978}]
{\label{lem:Montgomery}}
Let ${\cal{P}}$ denote the set of all primes.
Let ${\cal{B}}$ denote a set of integers that avoids $\omega_{\cal B}(p)$ residue classes modulo the prime $p$.
Here $\omega_{\cal B} : {\cal{P}} \rightarrow \N$
with $0 \leq \omega_{\cal B}(p) \leq p-1$.
Let ${\cal B}(x)$ denote the counting function ${\cal B}(x) =
\sum_{b \leq x, b \in {\cal B}} 1$.
Then the following upper bound holds, for any $Q$:
\[ {\cal B}(x) \leq \dfrac{x+Q^2}{L}, \text{ where }
L= \sum_{q \leq Q}
\mu^2(q) \prod_{p \mid q} \dfrac{\omega_{\cal B}(p)}{p-\omega_{\cal B}(p)}. \]
\end{lemma}

\begin{lemma}[Gallagher's larger sieve, \cite{Gallagher:1971}]
{\label{lem:Gallagherls}}
Let ${\cal P}_0$ denote any set of primes such that ${\cal{A}}$
lies modulo $p$ (for $p \in {\cal P}_0$) in at most $\nu_{\cal A}(p)$
residue classes. Then the following bound holds for ${\cal A}(N) :=
\sum_{a \leq N, a \in {\cal A}} 1$,
provided the denominator is positive:
\[ {\cal A}(N) \leq \dfrac{ - \log N + \sum_{p \in {\cal P}_0} \log p}{- \log N
+ \sum_{p \in {\cal P}_0} \dfrac{ \log p}{\nu_{\cal A}(p)} }.
\]
\end{lemma}

\subsection{The inverse sieve argument}{\label{sec:inversesieve}}

We are going to investigate sums $a+b$ that must consist of restricted prime
factors. Suppose that in a particular problem we must have $a+b \not\equiv 0
\bmod p$ whenever $a \in \mathcal{A}$ and $b \in \mathcal{B}$. Then a residue
class $a$ that occurs in the set ${\cal A}$ forbids a residue class $-a$ in
${\cal B}$, in other words (in sieve theory notation, see Lemmas \ref{lem:Montgomery}, \ref{lem:Gallagherls}) $\nu_{\cal A}(p)=\omega_{\cal B}(p)$. This means
that each residue class modulo $p$ can be used for a sieve process: either it
occurs in ${\cal A}$ and hence $-a\bmod p$ is forbidden in ${\cal B}$, which is
good when estimating ${\cal B}(x)$; or it does not occur in ${\cal A}$ and can be used when connecting $\nu_A(p)$ and ${\cal A}(x)$. It has proved
successful to conclude from ${\cal A}(x)$ to $\nu_{\cal A}(p)$ in some average
sense in the ``inverse sieve'' direction\footnote{The expression ``inverse
  sieve'' was
introduced in the first author's Habilitationsschrift 
\cite{Elsholtz:2002habil} to recall that in classical sieve problems one
deduces bounds on the counting function of a sequence from the number of sifted classes, whereas here one draws conclusions from the size of the counting function to the number of sifted classes.},
 using Gallagher's larger sieve, then to convert this into information on $\omega_{\cal B}(p)$ on average and use this to give an upper bound on ${\cal B}(x)$. (See for example \cite{Elsholtz:2001mathematika, Elsholtz:2002, Elsholtz:2006, Elsholtz:2008}).

It should be emphasized that this kind of argument is quite stable. It works
for a wide range of ${\cal A}(x)$ and it also works if only a subset of primes can be used for sieving, assuming these primes are somewhat regularly distributed so that the various sums of primes that arise in the argument get a fair share when restricted to the subset.

As a concrete example of (the first part of) this strategy we shall prove the following result, that will be needed several times:
\begin{lemma}{\label{InverseSieve}}
Let ${\cal P}_0$ be a subset of the primes. Let $\mathcal{A} \subset [1,x]$ be a set of integers that satisfies $\#\mathcal{A} = k \geq 2$, and that occupies $\nu_{\cal A}(p)$ residue classes modulo the prime $p$. Then for any $y \geq 10$ we have
\begin{equation}{\label{eq:lemmalowerboundonclasses}}
 \sum_{y/2 < p \leq y, p \in {\cal P}_0} \frac{\nu_{\cal A}(p)}{p} \geq
 k\sum_{y/2 < p \leq y, p \in {\cal P}_0} \frac{1}{p} - \frac{(k^{2}-k)\log
   x}{(y/2) \log(y/2)}.
\end{equation}
In particular, if $\sum_{y/2 < p \leq y, p \in {\cal P}_0} \log p \geq 8k\log
x$ then we have
$$ \sum_{y/2 < p \leq y, p \in {\cal P}_0} \frac{\nu_{\cal A}(p)}{p} \geq \frac{k}{2} \sum_{y/2 < p \leq y, p \in {\cal P}_0} \frac{1}{p} . $$
\end{lemma}

\begin{proof}[Proof of Lemma \ref{InverseSieve}]
Before we begin the proof, note that we may assume that
$$ \sum_{y/2 < p \leq y, p \in {\cal P}_0} \log p > \log x, $$
since otherwise we have
$$ k\sum_{y/2 < p \leq y, p \in {\cal P}_0} \frac{1}{p} \leq \frac{k}{(y/2) \log(y/2)} \sum_{y/2 < p \leq y, p \in {\cal P}_0} \log p \leq \frac{k \log x}{(y/2) \log(y/2)} $$
and the assertion ({\ref{eq:lemmalowerboundonclasses}}) of the lemma is trivial.

Applying Gallagher's larger sieve (that is Lemma \ref{lem:Gallagherls}), we see
$$ -\log x + \sum_{y/2 <p \leq y,  p \in {\cal P}_0} \frac{\log p}{\nu_{\cal A}(p)} \leq \frac{-\log x + \sum_{y/2 <p \leq y,  p \in {\cal P}_0} \log p}{\#\mathcal{A}} = \frac{-\log x + \sum_{y/2 <p \leq y,  p \in {\cal P}_0} \log p}{k}. $$
Strictly speaking this bound only follows from Lemma \ref{lem:Gallagherls} if the left hand side is strictly positive, but it is trivial otherwise because, by our initial remarks, we may assume that the right hand side is positive. Thus we see
$$ \sum_{y/2 <p \leq y,  p \in {\cal P}_0} \frac{\log p}{\nu_{\cal A}(p)} - \sum_{y/2 <p \leq y,  p \in {\cal P}_0} \frac{\log p}{k} \leq (1-1/k)\log x . $$

On the other hand, we certainly always have $\nu_{\cal A}(p) \leq \#\mathcal{A} = k$, so
\begin{eqnarray}
\sum_{y/2 <p \leq y,  p \in {\cal P}_0} \frac{\log p}{\nu_{\cal A}(p)} - \sum_{y/2 <p \leq y,  p \in {\cal P}_0} \frac{\log p}{k} & = & \sum_{y/2 <p \leq y,  p \in {\cal P}_0} \frac{(k - \nu_{\cal A}(p)) \log p}{k \nu_{\cal A}(p)} \nonumber \\
& \geq & \frac{(y/2) \log(y/2)}{k^{2}} \sum_{y/2 <p \leq y,  p \in {\cal P}_0} \frac{(k - \nu_{\cal A}(p))}{p}. \nonumber
\end{eqnarray}
It follows from the inequalities above that
\begin{eqnarray*}
 \sum_{y/2 < p \leq y, p \in {\cal P}_0} \frac{\nu_{\cal A}(p)}{p} 
&=& \sum_{y/2 < p \leq y, p \in {\cal P}_0} \frac{k}{p} - \sum_{y/2 < p \leq y,
  p \in {\cal P}_0} \frac{(k - \nu_{\cal A}(p))}{p} \\
&\geq &
\sum_{y/2 < p \leq y, p \in {\cal P}_0}
 \frac{k}{p} - \frac{(k^{2}-k)\log x}{(y/2) \log(y/2)}, \qquad (*)
\end{eqnarray*}
as claimed.

For the final assertion of the lemma, we just note that if $\sum_{y/2 <p \leq y,  p \in {\cal P}_0} \log p \geq 8k\log x$ and $y \geq 10$ then
$$ k \sum_{y/2 <p \leq y,  p \in {\cal P}_0} \frac{1}{p} \geq \frac{k}{y \log y} \sum_{y/2 <p \leq y,  p \in {\cal P}_0}  \log p \geq \frac{4k^{2} \log x}{y \log(y/2)}. $$
Thus the first term in $(*)$ is at least twice as large as the second.
\end{proof}

We remark that most of the effort in the proof of Lemma \ref{InverseSieve} goes into proving the first assertion, with the error term $2(k^{2}-k)(\log x)/(y \log(y/2))$. We will need this for the proofs of Propositions \ref{smallkscs} and \ref{smallkbv}, below.

\subsection{Ruzsa's inequality}
S\'{a}rk\"ozy \cite{Sarkozy:2012} recently showed that the following inequality of Ruzsa can be successfully applied to questions of ternary decomposability.

\begin{lemma}[Ruzsa \cite{Ruzsa:2009}, Theorem 9.1]{\label{lem:Ruzsa}}
Let ${\cal A,B,C}$ be finite sets in a commutative group. We have
\[|{\cal A+B+C}|^2 \leq |{\cal A+B}| |{\cal A+C}| |{\cal B+C}|.\]
\end{lemma}

S\'{a}rk\"ozy used this lemma in $\F_p$, but we
will see that Ruzsa's inequality provides a simple way to
deduce a ternary indecomposability result from a binary counting result over the integers too. For a
previous approach to such deductions, which used an inverse sieve and
made use of local knowledge modulo $p$, by means of
the Cauchy--Davenport theorem, see \cite{Elsholtz:2001mathematika}.

\subsection{Finite and infinite decomposability statements}{\label{sec:finiteinfinite}}
The results and conjectures stated in the introduction and section \ref{sec:announcement} refer to asymptotic decomposability of infinite sets, but in our proofs we will want to work with finite sets. The reduction to this situation is not difficult, but to avoid any possible obscurity we make some remarks about it now.

Suppose we are investigating whether or not ${\cal S}\sim{\cal A}+{\cal B}$. By definition, this would mean that
\[({\cal A+B})\cap [x_0, \infty)= {\cal S}\cap [x_0,\infty) , \]
where $x_{0}$ would be a fixed number depending on $\mathcal{A}$, $\mathcal{B}$ and $\mathcal{S}$. In particular, for any large $x$ we would have
\[{\cal S}\cap [x_0,x] \subseteq (\mathcal{A} \cap [0,x]) + (\mathcal{B} \cap [0,x]) \subseteq ({\cal S}\cap [x_0,2x]) \cup [0,x_{0}] . \]
In the problems that we study we know that elements of $\mathcal{S}$ are not divisible by certain primes, and this information can usually be applied straightforwardly to the ``finitised'' versions. For example, if we are investigating the smooth numbers $\mathcal{S}_{f(n)}$, where $f(n)$ is an increasing function, then all prime factors larger than $\max(x_{0},f(2x))$ are forbidden to occur in the sumset $(\mathcal{A} \cap [0,x]) + (\mathcal{B} \cap [0,x])$, and so can be used for sieving arguments. The other important point is that, for the infinite sets $\mathcal{S}$ that we are interested in, we have
$$ \#({\cal S}\cap [x_0,2x]) \ll \#{\cal S}\cap [x_0,x] $$
for all large $x$, with an absolute implied constant\footnote{This is clear when $\mathcal{S}$ is the set of primes, for example. For sets $\mathcal{S}_{f(n)}$ of smooth numbers it is true provided $f(2n) \leq f(n)(1+(100\log f(n))/\log n)$, say, as hypothesised in Theorem \ref{smooth1} and discussed further in section \ref{sec:proofsofthms}.}. Thus we know that $(\mathcal{A} \cap [0,x]) + (\mathcal{B} \cap [0,x])$ must make up a positive fraction of $({\cal S}\cap [x_0,2x]) \cup [0,x_{0}]$. It is really this assumption, rather than the assumption that they are identical, which is sufficient to make our finitary arguments run, and our general finitary Theorem \ref{genthm} explicitly covers the case where $\mathcal{A}+\mathcal{B}$ is equal to any reasonably large subset of the target set.

If $\mathcal{S} = \mathcal{P}$, the set of primes, then the sieving situation is slightly more complicated because each prime {\em will} divide precisely one of the elements of $\mathcal{P}$, namely itself. However, if we choose $x \geq x_{0}^{2}$, and suppose without loss of generality (for given $x$) that we have $\#(\mathcal{B} \cap [0,x]) \geq \#(\mathcal{A} \cap [0,x])$, then
\[ (\mathcal{A} \cap [0,x]) + (\mathcal{B} \cap [\sqrt{x},x]) \subseteq {\cal P}\cap [\sqrt{x},2x] . \]
Then all prime factors $p < \sqrt{x}$ are forbidden to occur in the sumset, and our sieving arguments could give an upper bound on $\#(\mathcal{B} \cap [\sqrt{x},x])$. The reader may check that all the upper bounds we obtain (in section \ref{sec:ostmannrevisited}) are $\gg \sqrt{x}$, so we have the same bounds (up to a multiplicative constant) for $\#(\mathcal{B} \cap [0,x])$.

In summary, in the following sections we will work with finite sets $\mathcal{S}$, and will investigate the possibility of 
additive decompositions $\mathcal{S} = \mathcal{A}+\mathcal{B}$, 
rather than inclusions $\mathcal{A}+\mathcal{B} \subseteq \mathcal{S}$
or asymptotic decompositions $\mathcal{A}+\mathcal{B} \sim \mathcal{S}$. We will also ignore any issues of omitting elements smaller than $x_{0}$ or $\sqrt{x}$. The reader may convince themselves that, in view of the above discussion, our results and arguments can easily be tweaked to yield the theorems actually claimed in section \ref{sec:announcement}.

\section{A general machinery}{\label{sec:genmachinery}}

\subsection{A general additive irreducibility theorem}
In this section we will develop some general machinery for obtaining additive
irreducibility results, from which the specific theorems in 
section \ref{sec:announcement}
 will follow on inputting suitable information about the specific sets $\mathcal{S}$, such as the smooth numbers, under study, and (to deduce the ternary indecomposability results) on using Ruzsa's inequality.

More precisely, we shall establish the following main result:
\begin{theorem}{\label{genthm}}
Let $x$ be large, let ${\cal P}_0$ be a subset of the primes, and suppose there exists $x^{-1/10} < c \leq 1$ such that
$$ \sum_{y/2 < p \leq y, p \in {\cal P}_0} \log p \geq c \sum_{y/2 < p \leq y}
\log p, \quad \text{ whenever } x^{1/10} \leq y \leq x^{1/2}. $$

Suppose $\mathcal{S} \subseteq [1,x]$ is a non-empty ``target set'' of integers none of which is divisible by any $p \in {\cal P}_0$, and let $\sigma := (\#\mathcal{S})/x$ denote the density of $\mathcal{S}$. Also suppose $\mathcal{S}_{0} \subseteq \mathcal{S}$ is any non-empty set, and let $\sigma_{0}:= (\#\mathcal{S}_{0})/(\#\mathcal{S})$ denote the {\em relative density} of $\mathcal{S}_{0}$.

Finally set $K := 1000 (\sigma_{0} \sigma c^{2})^{-1} \log^{2}x$ and ${\cal P}_0^{*} := {\cal P}_0 \cap [K^{3},\infty)$, and suppose that {\em at least one} of the following holds:
\begin{enumerate}
\item (Sieve Controls Size) we have the lower bound
$$ \sum_{\substack{1 < q \leq \sqrt{x}, \\ p \mid q \Rightarrow p \in {\cal P}_0^{*}}}
\mu^2(q) \prod_{p \mid q} \dfrac{2}{p} \geq 10 (\sigma_{0} \sigma)^{-1} ; $$

\item (Bombieri--Vinogradov) there exists a positive integer $Q$ such that
$$ \sum_{\substack{1 < q \leq Q, \\ p \mid q \Rightarrow p \in {\cal P}_0^{*}}}
\frac{\mu^2(q)}{q} \geq 10 \sigma_{0}^{-1} , $$
and also
$$ \sum_{\substack{d \leq Q^{2}, \\ p \mid d \Rightarrow p \in {\cal P}_0^{*}}} \mu^{2}(d) \tau_{3}(d)^{1+(\log K)/(\log 3)} \max_{(a,d)=1} \left|\#\{s \in \mathcal{S}: s \equiv a \; \textrm{mod } d \} - \frac{\#\mathcal{S}}{\phi(d)} \right| \leq \frac{\#\mathcal{S}}{2K\sigma_{0}^{-1}}, $$
where $\mu(d)$ denotes the M\"{o}bius function, $\phi(d)$ denotes the Euler totient function, and $\tau_{3}(d) := \#\{(u,v,w) \in \N^{3}: uvw=d\}$.
\end{enumerate}

Then if there is a decomposition $\mathcal{S}_{0} = \mathcal{A} + \mathcal{B}$, for some sets $\mathcal{A}, \mathcal{B}$ of non-negative integers with $2 \leq \#\mathcal{A} \leq \#\mathcal{B}$, one has the upper bound $\#\mathcal{B} \ll (\sqrt{x} \log^{4}x)/c^{4}$, and therefore also the lower bound $\#\mathcal{A} \gg (\sqrt{x} \sigma_{0} \sigma c^{4})/\log^{4}x$.
\end{theorem}

We remark that the reader will not lose much by thinking of Theorem \ref{genthm} in the simpler case where $\mathcal{S}_{0} = \mathcal{S}$, and therefore $\sigma_{0} = 1$. (And for simplicity we will assume this in all of our applications.) We allow ourselves the possibility of handling a subset $\mathcal{S}_{0}$ because, as discussed in section \ref{sec:finiteinfinite}, that is the situation that arises when carefully turning asymptotic irreducibility questions into finitary questions.

\vspace{12pt}
Note that the lower bound demanded in the ``Sieve Controls Size'' condition
involves a factor $(\sigma_{0} \sigma)^{-1}$, whereas the lower bound demanded in the
``Bombieri--Vinogradov condition'' only involves the smaller quantity $\sigma_{0}^{-1}$ (but in the Bombieri--Vinogradov
condition one also needs an upper bound for another sum, as in part 2 of the
theorem above). The ``Sieve Controls Size'' condition implies that one can more-or-less extract $\mathcal{S}_{0}$ by sieving the integers less than $x$ by the primes $p \in {\cal P}_0^{*}$, and ultimately that one can run a sieve process {\em on all those integers} to bound $\#\mathcal{A}$ and $\#\mathcal{B}$. If the ``Bombieri--Vinogradov condition'' is satisfied then one may not be able to succeed using a sieve process on $[1,x]$, but instead one can run a sieve process {\em inside $\mathcal{S}$}, using Selberg's sieve and other ingredients, to bound $\#\mathcal{A}$ and $\#\mathcal{B}$.

As we will see in 
sections \ref{sec:proofsofthms} -- \ref{sec:proofsprimefactors}, 
Theorem \ref{genthm} can be applied to sets of smooth numbers because they satisfy the ``Bombieri--Vinogradov condition'', and it can be applied to semigroup-type sets $\mathcal{S}$, whose elements are generated multiplicatively by a dense subsequence of the primes, because they satisfy the ``Sieve Controls Size'' condition. One could also use Theorem \ref{genthm} to study the primes themselves, which satisfy both conditions.

We will explain and prove Theorem \ref{genthm} in the next subsection, on the assumption of
some sieve theory type propositions, and then in 
section {\ref{subsec:proofsofprop}} we will prove those propositions. We hope that establishing a general result will clarify, as well as extend, the approach of previous papers\footnote{The previous work all roughly corresponds to various special applications of the ``Sieve Controls Size'' part of Theorem \ref{genthm}.} such as~\cite{Elsholtz:2001mathematika}.

\subsection{Outline proof of Theorem \ref{genthm}}
Suppose there is an additive decomposition $\mathcal{S}_{0} = \mathcal{A} + \mathcal{B}$, where $2 \leq \#\mathcal{A} \leq \#\mathcal{B}$. If we let $k:= \#\mathcal{A}$, then since $\#(\mathcal{A} + \mathcal{B}) \leq k \cdot \#\mathcal{B}$ we must certainly have
$$ \#\mathcal{B} \geq \frac{\#\mathcal{S}_{0}}{k} = \frac{\sigma_{0} \sigma x}{k}. $$
On the other hand, if the elements of $\mathcal{S}_{0} \subseteq \mathcal{S}$ are not divisible by certain primes then, for given $\mathcal{A}$, there will be congruence constraints modulo those primes on the possible elements of $\mathcal{B}$. We will use various sieve arguments to place upper bounds on $\#\mathcal{B}$, which will contradict the lower bound $\#\mathcal{B} \geq (\#\mathcal{S}_{0})/k $ unless $k$ and $\#\mathcal{B}$ are both ``fairly close'' to $\sqrt{x}$ (as asserted in Theorem \ref{genthm}).

\vspace{12pt}
In all of the results stated in this subsection, $\mathcal{S} \subseteq [1,x]$, ${\cal P}_0$ and ${\cal P}_0^{*} = {\cal P}_0 \cap [K^{3},\infty)$ will denote subsets of the integers and the primes, as in the statement of Theorem \ref{genthm}, and $0 \leq a_{1} < a_{2} < \dotsc < a_{k} \leq x$ will denote any distinct integers.

\vspace{12pt}
Firstly we present two results, corresponding to the ``Sieve Controls Size'' condition and the ``Bombieri--Vinogradov condition'', that can be useful when $k$ is assumed to be fairly small:
\begin{proposition}{\label{smallkscs}}
For any $2 \leq k \leq K =1000 (\sigma_{0} \sigma c^{2})^{-1} \log^{2}x$ we have the upper bound
$$ \#\{s \in \mathcal{S}: s \not\equiv a_{i} \; \textrm{ mod } p, 
\quad 1 \leq i \leq k, \; p \in {\cal P}_0^{*}\} \leq \frac{4x}{(k/2) \sum_{\substack{1 < q \leq \sqrt{x}, \\ p \mid q \Rightarrow p \in {\cal P}_0^{*}}}
\mu^2(q) \prod_{p \mid q} \dfrac{2}{p} }. $$
\end{proposition}

\begin{proposition}{\label{smallkbv}}
For any $2 \leq k \leq K =1000 (\sigma_{0} \sigma c^{2})^{-1} \log^{2}x$, and any positive integer parameter $Q$, $\#\{s \in
\mathcal{S}: s \not\equiv a_{i} \; \textrm{ mod } p, \quad 1 \leq i \leq k, \; p \in {\cal P}_0^{*}\}$ is at most
$$ \frac{2\#\mathcal{S}}{(k-1) \sum_{\substack{1 < q \leq Q, \\ p \mid q \Rightarrow p \in {\cal P}_0^{*}}}
\frac{\mu^2(q)}{q} } + \sum_{\substack{d \leq Q^{2}, \\ p \mid d \Rightarrow p \in {\cal P}_0^{*}}} \mu^{2}(d) \tau_{3}(d)^{1+\frac{\log k}{\log 3}} \max_{(a,d)=1} \left|\#\{s \in \mathcal{S}: s \equiv a \; \textrm{mod } d \} - \frac{\#\mathcal{S}}{\phi(d)} \right|. $$
\end{proposition}

At first glance it might seem that the case when $k$ is small should not be
difficult to handle. However, the situation is a little delicate, since there
(probably) {\em are} sets $\mathcal{S}$, very like the ones that we shall
study, that admit decompositions as $\mathcal{A} + \mathcal{B}$ with
$\#\mathcal{A} \geq 2$ small. For example, it is conjectured that the twin
primes constitute roughly a proportion $1/\log x$ of all the primes, and
they clearly admit a decomposition with $\mathcal{A} = \{0,2\}$. 
Also, the quadruples of primes $\{10k+1,3,7,9\}$ even admit a
ternary decomposition
$\{1,3\}+\{0,6\}+{\mathcal C}$. Here ${\mathcal C}$ is expected to be infinite.
Thus some work is really required, as in the above propositions, to know that we can perform useful sieving moves to handle small $k$ and get our argument started.

\vspace{12pt}
The next result will provide our main estimates when $k$ is larger: however, we state it only as a lemma in the first instance, since it will be incorporated into a slightly more general proposition immediately below.
\begin{lemma}{\label{middlek}}
Let $k$ be any natural number, and suppose that 
$$y_{1} < \frac{y_{2}}{2} <y_2 < \frac{\sqrt{x}}{y_{1}}$$ 
are large and are such that
$$ \sum_{y_{1}/2 < p \leq y_{1}, p \in {\cal P}_0} \log p \geq 8k \log x \;\;\; \textrm{ and } \;\;\; \sum_{y_{2}/2 < p \leq y_{2}, p \in {\cal P}_0} \log p \geq 8k \log x. $$
Then $\#\{s \in \mathcal{S}: s \not\equiv a_{i} \; \textrm{ mod } p, \quad 1 \leq i \leq k, \; p \in {\cal P}_0\}$ is 
$$ \leq \frac{128 x \log y_{1} \log y_{2}}{k^{2}} \cdot \frac{\sum_{y_{1}/2 < p \leq y_{1}} \log p}{\sum_{y_{1}/2 < p \leq y_{1}, p \in {\cal P}_0} \log p} \cdot \frac{\sum_{y_{2}/2 < p \leq y_{2}} \log p}{\sum_{y_{2}/2 < p \leq y_{2}, p \in {\cal P}_0} \log p}. $$
\end{lemma}

The slight problem with Lemma \ref{middlek} is that if $k$ is very large, e.g. larger than about $x^{1/4}$, then it is not possible to find values $y_{1},y_{2}$ satisfying the hypotheses. However, since the bound in the lemma decays quadratically with the number of values $a_{i}$ we can still obtain strong information if we artificially reduce the value of $k$, i.e. only work with a subset of the $a_{i}$. The following result, which also includes Lemma \ref{middlek}, records the kind of bound obtained:
\begin{proposition}{\label{middlek2}}
Suppose that $y_{1} < y_{2}/2 < y_2 < 
\sqrt{x}/y_{1}$ are large and are such that
$$ \sum_{y_{1}/2 < p \leq y_{1}, p \in {\cal P}_0} \log p \geq 8 \log x \;\;\; \textrm{ and } \;\;\; \sum_{y_{2}/2 < p \leq y_{2}, p \in {\cal P}_0} \log p \geq 8 \log x, $$
and let $k$ be arbitrary. Then $\#\{s \in \mathcal{S}: s \not\equiv a_{i} \; \textrm{ mod } p, \quad 1 \leq i \leq k, \; p \in {\cal P}_0\}$ is
$$ \leq 128 x M \log y_{1} \log y_{2} \cdot \frac{\sum_{y_{1}/2 < p \leq y_{1}} \log p}{\sum_{y_{1}/2 < p \leq y_{1}, p \in {\cal P}_0} \log p} \cdot \frac{\sum_{y_{2}/2 < p \leq y_{2}} \log p}{\sum_{y_{2}/2 < p \leq y_{2}, p \in {\cal P}_0} \log p}, $$
where
$$ M = M(y_{1},y_{2},x,k,{\cal P}_0) := \max \left\{\frac{1}{k^{2}}, \frac{256 \log^{2}x}{(\sum_{y_{1}/2 < p \leq y_{1}, p \in {\cal P}_0} \log p)^{2}},\frac{256 \log^{2}x}{(\sum_{y_{2}/2 < p \leq y_{2}, p \in {\cal P}_0} \log p)^{2}} \right\}. $$
\end{proposition}

\vspace{12pt}
The foregoing propositions can be applied to prove Theorem \ref{genthm}, as follows:
\begin{itemize}
\item by subtracting a suitable integer from all the elements of $\mathcal{A}$, and adding it to all the elements of $\mathcal{B}$, we may assume that we have $\mathcal{S}_{0}=\mathcal{A}+\mathcal{B}$ for some set $\mathcal{A}$ of non-negative integers that contains zero. Note this means, in particular, that $\mathcal{B} = \{0\}+\mathcal{B}$ is a subset of $\mathcal{S}_{0}$, and therefore a subset of $\mathcal{S}$, and in fact
$$ \mathcal{B} \subseteq \{s \in \mathcal{S} : s \not\equiv -a \; \textrm{ mod } p, \quad a \in \mathcal{A}, \; p \in {\cal P}_0\}. $$

\item recall that we write $k = \#\mathcal{A}$.

\item if $2 \leq k \leq 1000 (\sigma_{0} \sigma c^{2})^{-1} \log^{2}x$, and the ``Sieve Controls Size'' condition is satisfied, then the bound in Proposition \ref{smallkscs} is $\leq 4x/(5k \sigma_{0}^{-1} \sigma^{-1})$, which is $< x\sigma_{0}\sigma /k = (\#\mathcal{S}_{0})/k$. This contradicts our earlier observation that $\#\mathcal{B} \geq (\#\mathcal{S}_{0})/k$, so we conclude that $k$ cannot be so small.

\item if $2 \leq k \leq 1000 (\sigma_{0} \sigma c^{2})^{-1} \log^{2}x$, and the ``Bombieri--Vinogradov condition'' is satisfied, then the bound in Proposition \ref{smallkbv} is $\leq (2\#\mathcal{S})/(10\sigma_{0}^{-1}(k-1)) + (\#\mathcal{S})/(2K\sigma_{0}^{-1})$, which is at most $(9\#\mathcal{S})/(10\sigma_{0}^{-1}k) = (9\#\mathcal{S}_{0})/(10k)$. Again, $\mathcal{B}$ cannot be this small and so $k$ cannot be so small. 

\item if $k > 1000 (\sigma_{0} \sigma c^{2})^{-1} \log^{2}x$, we apply Proposition \ref{middlek2} with the choices $y_{1}=x^{1/4}/3$ and $y_{2}=x^{1/4}$, say. We have $\sum_{y_{1}/2 < p \leq y_{1}, p \in {\cal P}_0} \log p \geq c \sum_{y_{1}/2 < p \leq y_{1}} \log p \geq c x^{1/4}/10$, say (since $x$, and therefore $y_{1}$ and $y_{2}$, are large), similarly for the sum on the range $y_{2}/2 < p \leq y_{2}$, so that
\begin{eqnarray}
\#\mathcal{B} & \leq & \#\{s \in \mathcal{S} : s \not\equiv -a \; \textrm{ mod } p, \quad a \in \mathcal{A}, \; p \in {\cal P}_0\} \nonumber \\
& \leq & 128 x \max\{\frac{1}{k^{2}},\frac{25600 \log^{2}x}{c^{2} \sqrt{x}}\} \frac{\log^{2}x}{c^{2}}. \nonumber
\end{eqnarray}
Now we cannot have $\#\mathcal{B} \leq (128 x\log^{2}x)/(k^{2}c^{2})$, since this bound is $< (x \sigma_{0} \sigma)/k$ when $k$ is $> 1000 (\sigma_{0} \sigma c^{2})^{-1} \log^{2}x$, and we must have $\#\mathcal{B} \geq (\#\mathcal{S}_{0})/k = x\sigma_{0}\sigma /k$. Therefore we must have
$$ \#\mathcal{B} \ll x \cdot \frac{\log^{2}x}{c^{2} \sqrt{x}} \cdot \frac{\log^{2}x}{c^{2}} = \frac{\sqrt{x} \log^{4}x}{c^{4}}, $$
as asserted in Theorem \ref{genthm}. Since we have $\#\mathcal{A} \#\mathcal{B} \geq \#\mathcal{S}_{0} = x \sigma_{0}\sigma$, the lower bound claimed for $\#\mathcal{A}$ also immediately follows.
\end{itemize}
\qed

\subsection{Proofs of the propositions}{\label{subsec:proofsofprop}}

\begin{proof}[Proof of Proposition \ref{smallkscs}]
For each $p \in {\cal P}_0^{*}$ let $1 \leq \nu(p) \leq k$ denote the number of
distinct residue classes modulo $p$ occupied by the numbers $a_{1},a_{2},\dotsc
,a_{k}$. Then in view of Montgomery's Large Sieve (that is Lemma \ref{lem:Montgomery}), applied with the choice $Q=\sqrt{x}$, to prove Proposition \ref{smallkscs} it will suffice to show that
$$ \sum_{\substack{q \leq \sqrt{x}, \\ p \mid q \Rightarrow p \in {\cal P}_0^{*}}}
\mu^2(q) \prod_{p \mid q} \dfrac{\nu(p)}{p-\nu(p)} \geq (1/2) \cdot (k/2) \sum_{\substack{1 < q \leq \sqrt{x}, \\ p \mid q \Rightarrow p \in {\cal P}_0^{*}}}
\mu^2(q) \prod_{p \mid q} \dfrac{2}{p}, $$
provided $2 \leq k \leq K$. Recall that if $p \in {\cal P}_0^{*}$ then $p \geq K^{3}$, by definition.

Now note that, for any $q$ counted in the sum on the left hand side,
$$ \prod_{p \mid q} \dfrac{\nu(p)}{p-\nu(p)} = \prod_{p \mid q} \dfrac{(\nu(p)/p)}{1-(\nu(p)/p)} = \left( \prod_{p \mid q} \frac{\nu(p)}{p} \right) \left( \prod_{p \mid q} \sum_{i=0}^{\infty} \left(\frac{\nu(p)}{p} \right)^{i} \right), $$
so if we define a completely multiplicative function $f(n)$ by setting $f(p) =
\nu(p) \textbf{1}_{p \in {\cal P}_0^{*}}$, then
$$ \sum_{\substack{q \leq \sqrt{x}, \\ p \mid q \Rightarrow p \in {\cal P}_0^{*}}}
\mu^2(q) \prod_{p \mid q} \dfrac{\nu(p)}{p-\nu(p)} = \sum_{q \leq \sqrt{x}} \mu^2(q) \frac{f(q)}{q} \sum_{\substack{r=1, \\ p \mid r \Rightarrow p \mid q}}^{\infty} \frac{f(r)}{r} \geq \sum_{n \leq \sqrt{x}} \frac{f(n)}{n} . $$
Let us also define a completely multiplicative function $g(n)$, by setting $g(p) = k\textbf{1}_{p \in {\cal P}_0^{*}}$. Thus
$0 \leq f(p) \leq g(p) < p$ for each prime $p$. It is a general fact, whose proof we shall reproduce in just a moment, that for any non-negative completely multiplicative functions related in this way, 
$$ \sum_{n \leq \sqrt{x}} \frac{f(n)}{n} \geq \prod_{p \leq \sqrt{x}} \left(1-\frac{g(p)}{p}\right) \prod_{p \leq \sqrt{x}} \left(1-\frac{f(p)}{p}\right)^{-1} \sum_{n \leq \sqrt{x}} \frac{g(n)}{n} . \qquad (**) $$
Thus we have
$$ \sum_{\substack{q \leq \sqrt{x}, \\ p \mid q \Rightarrow p \in {\cal P}_0^{*}}}
\mu^2(q) \prod_{p \mid q} \dfrac{\nu(p)}{p-\nu(p)} \geq \sum_{n \leq \sqrt{x}}
\frac{f(n)}{n} \geq \prod_{\substack{p \leq \sqrt{x}, \\ 
p \in {\cal P}_0^{*}}} \left(1-\frac{(k-\nu(p))}{p-\nu(p)}\right)  \sum_{\substack{q \leq \sqrt{x}, \\ p \mid q \Rightarrow p \in {\cal P}_0^{*}}}
\mu^2(q) \prod_{p \mid q} \dfrac{k}{p} , $$
noting that the final sum over $q$ just consists of the squarefree terms from the corresponding sum $\sum_{n \leq \sqrt{x}} \frac{g(n)}{n}$.

Now if $p \in {\cal P}_0^{*}$ we certainly have $p \geq K^{3} \geq 2k \geq 2\nu(p)$, and therefore
$$ 0 \leq \frac{k-\nu(p)}{p-\nu(p)} \leq \frac{2(k-\nu(p))}{p} . $$
Then if we use the fact that $1-t \geq e^{-2t}$ when $0 \leq t \leq 1/4$, and
apply Lemma \ref{InverseSieve} on dyadic intervals (starting at $K^{3}$), we
see the product over primes in the previous display is
\begin{eqnarray}
 \prod_{\substack{p \leq \sqrt{x}, \\ 
p \in {\cal P}_0^{*}}} \left(1-\frac{(k-\nu(p))}{p-\nu(p)}\right)
&\geq &e^{-4\sum_{p \leq \sqrt{x}, p \in {\cal P}_0^{*}} \frac{k-\nu(p)}{p} } \nonumber\\
&\geq &e^{-(8(k^{2}-k)\log x)/(K^{3} \log(K^{3}))} \geq e^{-(8\log x)/K} \geq 1/2, \nonumber
\end{eqnarray}
say. Finally we note that
$$ \sum_{\substack{q \leq \sqrt{x}, \\ p \mid q \Rightarrow p \in {\cal P}_0^{*}}}
\mu^2(q) \prod_{p \mid q} \dfrac{k}{p} > \sum_{\substack{1 < q \leq \sqrt{x}, \\ p \mid q \Rightarrow p \in {\cal P}_0^{*}}}
\mu^2(q) \prod_{p \mid q} \dfrac{k}{p} \geq (k/2) \sum_{\substack{1 < q \leq \sqrt{x}, \\ p \mid q \Rightarrow p \in {\cal P}_0^{*}}}
\mu^2(q) \prod_{p \mid q} \dfrac{2}{p}, $$
since each term in the second sum involves at least one factor $k=(k/2) \cdot 2$. This all suffices to prove Proposition \ref{smallkscs}.

It remains to prove the general inequality $(**)$, where $f$ and $g$ are completely multiplicative functions satisfying $0 \leq f(p) \leq g(p) < p$. The reader may find this result in the forthcoming book~\cite{GranvilleandSound} of Granville and Soundararajan, but we will reproduce the short proof here. Let $h$ be the multiplicative function defined by
$$ h(p^{b}) = (g(p)-f(p))g(p^{b-1}), \;\;\; p \leq \sqrt{x}, \;\;\; b \geq 1, $$
and by $h(p^{b}) = 0$ when $p > \sqrt{x}$. Thus $h \geq 0$, and the reader may check that $g=f*h$, where $*$ here denotes Dirichlet convolution. Then
\begin{eqnarray}
\sum_{N \leq \sqrt{x}} \frac{g(N)}{N} = \sum_{mn \leq \sqrt{x}} \frac{h(m) f(n)}{mn} & \leq & (\sum_{m=1}^{\infty} \frac{h(m)}{m}) \sum_{n \leq \sqrt{x}} \frac{f(n)}{n} \nonumber \\
& = & \prod_{p \leq \sqrt{x}} \left(1-\frac{g(p)}{p}\right)^{-1} \prod_{p \leq \sqrt{x}} \left(1-\frac{f(p)}{p}\right) \sum_{n \leq \sqrt{x}} \frac{f(n)}{n}, \nonumber
\end{eqnarray}
as claimed in $(**)$.
\end{proof}

\begin{proof}[Proof of Proposition \ref{smallkbv}]
The proof is like that of Proposition \ref{smallkscs}, but crucially using Selberg's
sieve instead of the large sieve. Let $0 \leq \nu(p) \leq k$ denote the number
of {\em non-zero} residue classes modulo $p$ occupied by the numbers
$a_{1},a_{2},\dotsc ,a_{k}$, noting this is a slightly different definition than previously. Applying Lemma \ref{lem:Selberg} with the choice $\omega_{\mathcal{B}}(p)=\frac{p\nu(p)}{p-1}$, (and ${\cal P}_0$ replaced by ${\cal P}_0^{*}$, and $\mathcal{C}$ replaced by $\mathcal{S}$, and recalling that $p \geq K^{3}$ if $p \in {\cal P}_0^{*}$), we find $\#\{s \in \mathcal{S}: s \not\equiv a_{i} \; \textrm{ mod } p, \quad 1 \leq i \leq k, \; p \in {\cal P}_0^{*}\}$ is at most
$$ \frac{\#\mathcal{S}}{\sum_{\substack{q \leq Q, \\ p \mid q \Rightarrow p \in {\cal P}_0^{*}}}
\mu^2(q) \prod_{p \mid q} \dfrac{\nu(p)}{p-\nu(p)}} + \sum_{\substack{d \leq Q^{2}, \\ p \mid d \Rightarrow p \in {\cal P}_0^{*}}} \mu^{2}(d) \tau_{3}(d) \left|\#\{s \in \mathcal{S}: d \mid \prod_{i=1}^{k}(s-a_{i}) \} - \#\mathcal{S} \prod_{p \mid d} \frac{\nu(p)}{(p-1)} \right|. $$
Arguing exactly as in the proof of Proposition \ref{smallkscs}, except choosing $g(p)=k-1$ (instead of $k$) for those primes $p \in {\cal P}_0^{*}$ such that one of $a_{1},\dotsc ,a_{k}$ is congruent to zero modulo $p$, one finds that
$$ \sum_{\substack{q \leq Q, \\ p \mid q \Rightarrow p \in {\cal P}_0^{*}}}
\mu^2(q) \prod_{p \mid q} \dfrac{\nu(p)}{p-\nu(p)} \geq \frac{(k-1)}{2} \sum_{\substack{1 < q \leq Q, \\ p \mid q \Rightarrow p \in {\cal P}_0^{*}}}
\frac{\mu^2(q)}{q} . $$
Thus the first term in the previous display is at most as large as the first term in the statement of Proposition \ref{smallkbv}. It remains to treat the second term.

But a squarefree integer $d$ divides
$\prod_{i=1}^{k}(s-a_{i})$ if and only if, for each prime divisor $p$ of $d$,
$s$ lies modulo $p$ in one of the residue classes occupied by the numbers 
$a_{1},\dotsc ,a_{k}$. By assumption, $s \not\equiv 0 \; \textrm{ mod } p$ for
any $s \in \mathcal{S}$ and $p \in {\cal P}_0^{*}$, so actually $s$ must lie in
one of the $\nu(p)$ non-zero residue classes occupied by $a_{1},\dotsc ,a_{k}$. By the Chinese Remainder Theorem, this is the same as saying that $s$ lies modulo $d$ in one of $\prod_{p|d} \nu(p)$ residue classes, so the triangle inequality implies the second term is
$$ \leq \sum_{\substack{d \leq Q^{2}, \\ p \mid d \Rightarrow p \in {\cal P}_0^{*}}} \mu^{2}(d) \tau_{3}(d) \left(\prod_{p|d} \nu(p) \right) \max_{(a,d)=1} \left|\#\{s \in \mathcal{S}: s \equiv a \; \textrm{mod } d \} - \frac{\#\mathcal{S}}{\phi(d)} \right|. $$
Since $\prod_{p|d} \nu(p) \leq \prod_{p|d} k = (\prod_{p|d} 3)^{(\log k)/(\log 3)}$, and $\prod_{p|d} 3$ is at most $\tau_{3}(d)$, the bound asserted in Proposition \ref{smallkbv} follows.
\end{proof}

\begin{proof}[Proof of Lemma \ref{middlek}]
This is a consequence of Montgomery's Large Sieve and an ``inverse sieve''
argument, similar to (but easier than) Proposition \ref{smallkscs}. For each $p
\in {\cal P}_0$, let $1 \leq \nu(p) \leq k$ denote the number of distinct
residue classes modulo $p$ occupied by the numbers $a_{1},a_{2},\dotsc ,a_{k}$. Using Lemma \ref{lem:Montgomery} with the choice $Q=\sqrt{x}$, we see $\#\{s \in \mathcal{S}: s \not\equiv a_{i} \; \textrm{ mod } p, \quad 1 \leq i \leq k, \; p \in {\cal P}_0\}$ is
$$ \leq \frac{2x}{\sum_{q \leq \sqrt{x}, \atop p \mid q \Rightarrow p \in {\cal P}_0}
\mu^2(q) \prod_{p \mid q} \dfrac{\nu(p)}{p-\nu(p)}} \leq \frac{2x}{(\sum_{y_{1}/2 < p \leq y_{1}, p \in {\cal P}_0} \dfrac{\nu(p)}{p-\nu(p)}) (\sum_{y_{2}/2 < p \leq y_{2}, p \in {\cal P}_0} \dfrac{\nu(p)}{p-\nu(p)})}, $$
since each product $p_{1}p_{2}$ with $y_{1}/2 < p_{1} \leq y_{1}$ and $y_{2}/2 < p_{2} \leq y_{2}$ is a squarefree integer that is $\leq \sqrt{x}$.

The above is at most $2x/((\sum_{y_{1}/2 < p \leq y_{1}, p \in {\cal P}_0} \nu(p)/p) (\sum_{y_{2}/2 < p \leq y_{2}, p \in {\cal P}_0} \nu(p)/p))$, and by Lemma \ref{InverseSieve} each of the terms in the denominator is at least
$$ \frac{k}{2} \sum_{y_{i}/2 < p \leq y_{i}, p \in {\cal P}_0} \frac{1}{p} \geq \frac{k}{2y_{i} \log y_{i}} \sum_{y_{i}/2 < p \leq y_{i}, p \in {\cal P}_0} \log p \geq \frac{k}{4\log y_{i}} \frac{\sum_{y_{i}/2 < p \leq y_{i}, p \in {\cal P}_0} \log p}{\sum_{y_{i}/2 < p \leq y_{i}} \log p} \sum_{y_{i}/2 < p \leq y_{i}} \frac{\log p}{p}. $$
Lemma \ref{middlek} follows on noting that, by standard results on the distribution of prime numbers, $\sum_{y_{i}/2 < p \leq y_{i}} (\log p)/p \geq 1/2$ when $y_{i}$ is large.
\end{proof}

\begin{proof}[Proof of Proposition \ref{middlek2}]
Define
$$ k_{0} := \min\{k, \lfloor (\sum_{y_{1}/2 < p \leq y_{1}, p \in {\cal P}_0} \log p)/ 8\log x \rfloor, \lfloor (\sum_{y_{2}/2 < p \leq y_{2}, p \in {\cal P}_0} \log p)/ 8\log x \rfloor \}, $$
where $\lfloor \cdot \rfloor$ denotes the
integer part, noting that $k_{0} \geq 1$ by the hypotheses of Proposition \ref{middlek2}. Then we trivially have that
$$ \#\{s \in \mathcal{S}: s \not\equiv a_{i} \; \textrm{ mod } p, \quad 1 \leq i \leq k, \; p \in {\cal P}_0\} \leq \#\{s \in \mathcal{S}: s \not\equiv a_{i} \; \textrm{ mod } p, \quad 1 \leq i \leq k_{0}, \; p \in {\cal P}_0\}, $$
and Proposition \ref{middlek2} follows if we use Lemma \ref{middlek} to bound the right hand side (and observe that $\lfloor t \rfloor^{2} \geq t^{2}/4$ for all $t \geq 1$).
\end{proof}

\section{Proofs of Theorems \ref{smooth1} and \ref{smooth2}, application to smooth numbers}{\label{sec:proofsofthms}}
We shall prove Theorem \ref{smooth1} first, by feeding suitable information about the set of $y$-smooth numbers into our general additive irreducibility result, Theorem \ref{genthm}. Corollary \ref{smooth2} will follow from Theorem \ref{smooth1} by a general type of argument using Ruzsa's inequality.

More specifically, for any $x,y \geq 2$ and any positive integers $a,d$ let us
recall the following standard notation for smooth numbers:
$$ \Psi(x,y) := \mathcal{S}_{y}(x) = \#\{n \leq x : p \nmid n, \; p > y\}, $$
$$ \Psi_{d}(x,y) := \#\{n \leq x : (n,d)=1, \; p \nmid n, \; p > y\}, $$
$$ \Psi(x,y;a,d) := \#\{n \leq x : n \equiv a \; \textrm{mod } d, \; p \nmid n, \; p > y\}. $$
 We shall require the
following two results: a lower bound, and a Bombieri--Vinogradov type result.

\begin{theorem}[See e.g. Corollary 7.9 of Montgomery and Vaughan~\cite{MontgomeryandVaughan:2007}]\label{smoothsize}
For any fixed $a \geq 1$ and any $\log^{a}x \leq y \leq x$ we have
$$ \Psi(x,y) \geq \Psi(x,\log^{a}x) = x^{1-1/a + o_{a}(1)}, $$
where the term $o_{a}(1)$ tends to $0$ as $x \rightarrow \infty$.
\end{theorem}

\begin{theorem}[See Proposition 2 of Harper~\cite{Harper:2012}]\label{smoothprop2}
Let $0 < \eta \leq 1/80$ be any fixed constant, and let ${\cal P}_0 \subseteq [x^{\eta},\infty)$ be a set of primes. Then for any large $y \leq x$ and any $x^{\eta} \leq Q \leq \sqrt{x}$ we have
$$ \sum_{\substack{d \leq Q, \\ p \mid d \Rightarrow p \in {\cal P}_0}} \max_{(a,d)=1} \left|\Psi(x,y;a,d) - \frac{\Psi_{d}(x,y)}{\phi(d)} \right| \ll \log^{7/2}x \sqrt{\Psi(x,y)} (Q + x^{1/2-\eta}\log^{2}x) , $$
where $\phi(d)$ is Euler's totient function.
\end{theorem}

Theorem \ref{smoothprop2} follows on noting that the left hand side is at most
$$ \sum_{\substack{d \leq Q, \\ p \mid d \Rightarrow p \in {\cal P}_0}} \frac{1}{\phi(d)} \sum_{\substack{\chi \; (\textrm{mod } d), \\ \chi \neq \chi_{0}}} \left|\sum_{\substack{n \leq x, \\ n \in \mathcal{S}_{y}}} \chi(n) \right| \leq \sum_{x^{\eta} \leq r \leq Q} \sum_{\substack{\chi^{*} \; (\textrm{mod } r), \\ \chi^{*} \; \textrm{primitive}}} \sum_{x^{\eta} \leq d \leq Q} \frac{1}{\phi(d)} \sum_{\substack{\chi \; (\textrm{mod } d), \\ \chi^{*} \; \textrm{induces } \chi}} \left|\sum_{\substack{n \leq x, \\ n \in \mathcal{S}_{y}}} \chi(n) \right|, $$
and then using Proposition 2 of Harper~\cite{Harper:2012} (and the discussion in the two paragraphs following Proposition 2) to bound this. Note
that because we are only interested in moduli $d$ with all their prime factors
larger than $x^{\eta}$, we do not need to include values $r$ of the conductor that are smaller than $x^{\eta}$, which allows us to give quite a strong bound. See also Fouvry and Tenenbaum's paper~\cite{FouvryandTenenbaum:1996} for an earlier Bombieri--Vinogradov type result for smooth numbers, which would have allowed us to prove an additive irreducibility result on the smaller range $x^{(\log\log\log x)/\log\log x} \leq y \leq x^{\kappa}$.

Note that the counting function $\Psi(x,y)$ is very small when $y$ is small (and in fact even when $y$ is quite large), but that the set of prime factors forbidden to occur in the $y$-smooth numbers, namely the primes between $y$ and $x$, is not a very large set as measured by the kinds of denominators occurring in sieve bounds. In particular, we cannot satisfy the ``Sieve Controls Size'' condition, or anything similar, in this case, but instead we can use Theorem \ref{smoothprop2} to satisfy the ``Bombieri--Vinogradov condition'' in Theorem \ref{genthm}.

\begin{proof}[Proof of Theorem \ref{smooth1}]
Recall from the statement that $\kappa$ is a sufficiently small absolute constant,
and $D$ is a sufficiently large absolute constant.

We will apply Theorem \ref{genthm} to the set $\mathcal{S} = \mathcal{S}_{y} \cap [0,x]$, where we put $y = f(x)$, aiming to deduce\footnote{Recall the discussion of finite and infinite decomposability statements in section \ref{sec:finiteinfinite}. Since we assume in Theorem \ref{smooth1} that $f(n)$ is increasing, $f(2n) \leq f(n)(1+(100\log f(n))/\log n)$, and that $f(n) \geq \log^{2}n$, it is a fact that
$$ \# \mathcal{S}_{f(n)} \cap [0,x] \geq \# \mathcal{S}_{f(x/2)} \cap [x/2,x] \gg \# \mathcal{S}_{f(x/2)} \cap [0,x] \gg \#\mathcal{S}_{y} \cap [0,x] , $$
with an absolute implied constant. Thus we can deduce Theorem \ref{smooth1} from results about the finitary case $\mathcal{S}=\mathcal{A}+\mathcal{B}$. (For the final inequality see the proof of Smooth Numbers Result 2 in \cite{Harper:2012}, which can be adapted to show that if $y \geq \log^{2}x$, say, then $\Psi(x,y(1+(100\log y)/\log x)) \ll \Psi(x,y)$ (since if $y \geq \log^{2}x$ then $y/\log x \geq \sqrt{y}$, so one can save a factor of $\log y$ in the estimation of $\sum_{y < p \leq y(1+(100\log y)/\log x)}$ there).)} that if $\mathcal{S}=\mathcal{A}+\mathcal{B}$ then we must have $\#\mathcal{A}, \#\mathcal{B} \ll \sqrt{x} \log^{4}x$.

Since, by assumption, we have $\log^{D}x \leq y \leq x^{\kappa} \leq x^{1/10}/2$ (say), we can take ${\cal P}_0 = \{x^{\kappa} < p \leq x : p \textrm{ prime}\}$ and $c=1$ in Theorem \ref{genthm}. Moreover, using Theorem \ref{smoothsize} we see
$$ \sigma = \frac{\Psi(x,y)}{x} \geq \frac{\Psi(x,\log^{D}x)}{x} = x^{-1/D + o(1)} , $$
which is at least $1000 x^{-2/D} \log^{2}x$ (say) provided $x$ is large enough. Thus we have $K = 1000\sigma^{-1}\log^{2}x \leq x^{2/D}$ in Theorem \ref{genthm}, so that ${\cal P}_0^{*} = {\cal P}_0$ in this case (provided $x$ is large enough and $6/D \leq \kappa$, which we may certainly assume).

Now it suffices to check that the two parts of the ``Bombieri--Vinogradov condition'' in Theorem \ref{genthm} are satisfied. We choose $Q=x^{1/5}$ and note that, in view of a standard result on the distribution of prime numbers (and since $\kappa$ is small),
$$ \sum_{\substack{1 < q \leq Q, \\ p \mid q \Rightarrow p \in {\cal P}_0^{*}}}
\frac{\mu^2(q)}{q} \geq \sum_{x^{\kappa} < p \leq x^{1/5}} \frac{1}{p} = \log\log(x^{1/5}) - \log\log(x^{\kappa}) + O(1) \geq 10, $$
as required.

Next, if $d \leq Q^{2} = x^{2/5}$ has all of its prime factors from ${\cal P}_0^{*} = {\cal P}_0$ then it has at most $\log(x^{2/5})/\log(x^{\kappa}) < \kappa^{-1}$ prime factors, and therefore
$$ \tau_{3}(d) \leq 3^{1/\kappa},$$
since $\tau_{3}(d)$ is at most the number of ways of partitioning the prime
factors of $d$ into three sets. Thus
\begin{eqnarray}
&& \sum_{\substack{d \leq Q^{2}, \\ p \mid d \Rightarrow p \in {\cal P}_0^{*}}} \mu^{2}(d) \tau_{3}(d)^{1+(\log K)/(\log 3)} \max_{(a,d)=1} \left|\Psi(x,y;a,d) - \frac{\Psi(x,y)}{d} \right| \nonumber \\
& \leq & K^{C} \sum_{\substack{d \leq x^{2/5}, \\ p \mid d \Rightarrow p \in {\cal P}_0^{*}}} \mu^{2}(d) \max_{(a,d)=1} \left|\Psi(x,y;a,d) - \frac{\Psi_{d}(x,y)}{\phi(d)} \right| \nonumber
\end{eqnarray}
for a certain constant $C=C(\kappa)$, using the fact that $\Psi(x,y)=\Psi_{d}(x,y)$ if all the prime factors of $d$ are larger than $y$. Recalling that we have $K \leq x^{2/D} $, where $D$ is suitably large in terms of $\kappa$, then Theorem \ref{smoothprop2} shows the above is $\leq \Psi(x,y)/2K$, and so the Bombieri--Vinogradov condition is satisfied.
\end{proof}

\begin{proof}[Proof of Corollary \ref{smooth2}]
Suppose, for a contradiction, that we had $\mathcal{S}_{y} \cap [0,x] = \mathcal{A} + \mathcal{B} + \mathcal{C}$ for sets $\mathcal{A}, \mathcal{B}, \mathcal{C}$ of positive integers each containing at least two elements. By Ruzsa's inequality, (that is Lemma \ref{lem:Ruzsa}),
$$ \Psi(x,y)^{2} = \mathcal{S}_{y}(x)^{2} \leq (\#(\mathcal{A}+\mathcal{B}))(\#(\mathcal{A}+\mathcal{C})) (\#(\mathcal{B}+\mathcal{C})). $$
But if $\mathcal{S}_{y} \cap [0,x] = \mathcal{A} + \mathcal{B} + \mathcal{C} = (\mathcal{A} + \mathcal{B}) + \mathcal{C}$ then, by Theorem \ref{smooth1}, $\#(\mathcal{A}+\mathcal{B}) \ll \sqrt{x} \log^{4}x$, and similarly we would have $\#(\mathcal{A}+\mathcal{C}), \#(\mathcal{B}+\mathcal{C}) \ll \sqrt{x} \log^{4}x$. Thus
$$ \Psi(x,y)^{2} \ll x^{3/2} \log^{12}x , $$
which massively contradicts our lower bound $\Psi(x,y) \geq \Psi(x,\log^{D}x) = x^{1-1/D + o(1)}$ provided $D$ and $x$ are large enough.  
\end{proof}

\section{Proofs of Theorems \ref{thm:tauupperbound} and \ref{cor:semigroup}, sets generated by a dense subsequence of the primes}
{\label{sec:proofsprimefactors}}
In this section we shall prove Theorem \ref{thm:tauupperbound} and Corollary \ref{cor:semigroup}.

We first state Wirsing's mean value theorem, as this is used
in counting the elements in ${\cal Q}({\cal T})$.
\begin{lemma}[Wirsing \cite{Wirsing:1967}]
Let $f$ be a non-negative multiplicative arithmetic function that satisfies the
following hypotheses:
\begin{enumerate}
\item
\[\sum_{p \leq x} f(p) \frac{\log p}{p} \sim \tau \log x,\]
where $\tau$ is a positive constant.
\item
$f(p)=O(1)$.
\item
\[ \sum_{p \text{ prime, } n \geq 2}
\frac{f(p^n)}{p^n}< \infty.\]
\item
  if $\tau \leq 1$,
\[ \sum_{\substack{p \text{ prime, } n \geq 2, \\ p^{n} \leq x}} f(p^n)\ll \frac{x}{\log x}.\]
\end{enumerate}
 Then 
\[ \sum_{n \leq x} f(n) = (1+o(1)) \frac{x}{\log x} 
\frac{e^{-\gamma \tau}}{\Gamma(\tau)} 
\prod_{p \leq x} 
\left( 1+\frac{f(p)}{p}+ \frac{f(p^2)}{p^2}+ \dotsc \right).\\
\]
Here $\gamma$ is Euler's constant.

\end{lemma}

We now evaluate the counting function ${\cal Q}({\cal T})(x)$ (which is quite standard, see also
\cite{Elsholtz:2006}). In fact, for our purposes it would basically suffice to have a reasonable lower bound for ${\cal Q}({\cal T})(x)$, but since Wirsing's theorem readily supplies an asymptotic we will compute this.
\begin{lemma}
\[{\cal Q}({\cal T})(x)\sim C_{\cal{T}}\frac{x}{(\log x)^{1-\tau}}, \]
where $C_{\cal{T}} > 0$ is a constant depending on $\cal{T}$.
\end{lemma}
\begin{proof}[Sketch proof of Lemma]
The regularity condition on ${\cal T}$ in Theorem \ref{thm:tauupperbound} guarantees that Wirsing's theorem on sums of multiplicative functions can be applied.
We apply it to the multiplicative function
$f(p^n)= 
\begin{cases}1 & \text{ if }p \in {\cal T}\\
0& \text{ otherwise,}
\end{cases}$ 
so that
\[{\cal Q}({\cal T})(x)= \sum_{n \leq x} f(n).\] 
Here
\[\sum_{p \leq x} \frac{f(p) \log p}{p} = \tau \log x +C +o(1), \]
and all further hypotheses of Wirsing's theorem are satisfied as well. Thus

\begin{eqnarray}
\sum_{n \leq x} f(n)&=&(1+o(1)) \frac{x}{\log x} 
\frac{e^{-\gamma \tau}}{\Gamma(\tau)} 
\prod_{p \leq x} 
\left( 1+\frac{f(p)}{p}+ \frac{f(p^2)}{p^2}+ \dotsc \right) \nonumber\\
& = & (1+o(1))\frac{x}{\log x} \frac{e^{-\gamma \tau}}{\Gamma(\tau)}
\exp \left( \sum_{p \leq x, p \in {\cal T}} \log (1+\frac{1}{p-1})\right)
\nonumber \\
&\sim& C_{\cal T} \frac{x}{\log x} \exp (\tau \log \log x)\nonumber \\
&\sim & C_{\cal T} \frac{x}{(\log x)^{1-\tau}}.\nonumber
\end{eqnarray}
Here we used that $\sum_{p \leq x, p \in {\cal T}} \log (1+\frac{1}{p-1}) = \sum_{p \leq x, p \in {\cal T}} 1/p + C_{\cal T}' + o(1)$, and that from 
$\sum_{p \leq x, p \in {\cal T}} \frac{\log p}{p} = \tau \log x 
+C+o(1)$ it follows by partial summation that
$\sum_{p \leq x, p \in {\cal T}} 1/p = \tau \log \log x +C_{\cal T}''+o(1)$.

\end{proof}

\begin{proof}[Proof of Theorem \ref{thm:tauupperbound}.]
We intend to apply Theorem \ref{genthm} (1), putting $\mathcal{S} = {\cal Q}({\cal T}) \cap [1,x]$ and with ${\cal P}_0 =({\cal P} \setminus {\cal T})$ as the set of sieving primes. In view of the foregoing calculation we have $\sigma = {\cal Q}({\cal T})(x)/x = (1+o(1))C_{\cal T} \log^{\tau-1}x$ in the theorem. Also, in view of the regularity condition on ${\cal T}$ (and partial summation) we have
\begin{eqnarray}
\sum_{y/2 < p \leq y, p \in {\cal P}_0} \log p 
= (1+o(1))\frac{y}{2} - \sum_{y/2 < p \leq y, p \in {\cal T}} \log p & = & (1+o(1))(1-\tau)\frac{y}{2} \nonumber \\
& = & (1+o(1))(1-\tau)\sum_{y/2 < p \leq y} \log p . \nonumber
\end{eqnarray}
Hence we choose $c=(1+o(1))(1-\tau)$ in Theorem \ref{genthm}.

If we set ${\cal P}_0^*=({\cal P} \setminus {\cal T})\cap [K^3 , \infty)$, where $K=1000(\sigma c^2)^{-1} (\log x)^2$; and we define a completely multiplicative function $f(n)$ by $f(p)= 
\begin{cases}2 & \text{ if }p \in {\cal P}_0^* \\ 
0& \text{ otherwise}
\end{cases}$; and we let $d(n)$ denote the divisor function; then
\begin{eqnarray}
\sum_{\substack{1 < q \leq \sqrt{x}, \\ p \mid q \Rightarrow p \in {\cal P}_0^{*}}} \mu^2(q) \prod_{p \mid q} \dfrac{2}{p} \gg \sum_{q \leq \sqrt{x}} \frac{f(q)}{q} & \gg & \prod_{2 < p \leq \sqrt{x}} \left(1-\frac{2}{p}\right) \prod_{p \leq \sqrt{x}} \left(1-\frac{f(p)}{p}\right)^{-1} \sum_{n \leq \sqrt{x}} \frac{d(n)}{n} \nonumber \\
& \gg & \prod_{p \leq \sqrt{x}} \left(1-\frac{f(p)}{p}\right)^{-1} \nonumber \\
& \gg & \exp(\sum_{K^{3} \leq p \leq \sqrt{x}, p \in {\cal P} \setminus {\cal T}} \frac{2}{p}) \nonumber \\
& \gg & ((\log x)/\log\log x)^{2 -2\tau}. \nonumber
\end{eqnarray}
In the first inequality here we added the non-squarefree values of $q$ to the
sum, (which, as the reader may easily convince themselves, does not
 change the order of magnitude), 
and in the second inequality we used the general lower bound (**), from the proof of Proposition \ref{smallkscs}.

We conclude that the left hand side is much larger than $\sigma^{-1} \asymp \log^{1-\tau}x$, and so the ``Sieve Controls Size'' condition in Theorem \ref{genthm} is satisfied. Applying that theorem, we immediately obtain the upper bound
\[\max ({\cal A}(x), {\cal B}(x)) \ll_{\cal{T}} \sqrt{x} (\log x)^4.\]
\end{proof}

\begin{proof}[Proof of Corollary \ref{cor:semigroup}]
The proof of the corollary is as before. Suppose that ${\cal A}+{\cal B}+{\cal C}\sim {\cal Q}({\cal T})$.
Then by the binary result
 $({\cal A}+{\cal B})(x),({\cal A}+{\cal C})(x),({\cal B}+{\cal C})(x)
\sim x^{1/2+o(1)}$, and so by Ruzsa's inequality (Lemma \ref{lem:Ruzsa})
\[\frac{x^2}{(\log x)^{2-2\tau}}\ll_{\mathcal{T}} {\cal Q}({\cal T})(x)^{2} \ll \left(({\cal A}+{\cal B}+{\cal C})(x)\right)^2\leq x^{3/2+o(1)},\]
a contradiction.
\end{proof}

\section{Proof of Theorem \ref{Ostmann}, Ostmann's problem revisited}{\label{sec:ostmannrevisited}}
In this section we shall prove Theorem \ref{Ostmann}. In view of the existing results on Ostmann's problem (see \cite{Elsholtz:2001mathematika}, for example), we may assume that in a putative decomposition of the primes we have, for given $x$,
$$ \frac{\sqrt{x}}{\log^{5}x} \leq \mathcal{A}(x) \leq \mathcal{B}(x), $$
say. We will show that, in order for the lower bound on $\mathcal{A}(x)$ not to be violated, we must have $\mathcal{B}(x) \ll \sqrt{x} \log\log x$, which (since we must have $\mathcal{A}(x) \mathcal{B}(x) \gg \pi(x)$) immediately implies the claimed result that
$$ \frac{\sqrt{x}}{\log x \log\log x} \ll \mathcal{A}(x) \leq \mathcal{B}(x) \ll \sqrt{x} \log\log x. $$
We will do this by performing a bit more delicate analysis, requiring a lower bound estimate of Hildebrand~\cite{Hildebrand:1987}, of the final applications of the large and larger sieves in our general proofs.

Suppose, then, that for each prime $p \leq Y:= \sqrt{x}/\log^{10}x$ the set $\mathcal{A}$ occupies $\nu_{\cal A}(p) = p/2 + \epsilon_{p}$ residue classes modulo $p$. This leaves $\nu_{\cal B}(p) \leq p/2 - \epsilon_{p}$ residue classes for the set $\mathcal{B}$ to occupy, since none of the sums $a+b$ may be divisible by $p$. Now
\begin{eqnarray}
\sum_{p \leq Y} \frac{\log p}{p/2+\epsilon_{p}} + \sum_{p \leq Y} \frac{\log p}{p/2-\epsilon_{p}} & = & \sum_{p \leq Y} \frac{p \log p}{(p/2)^{2} - \epsilon_{p}^{2}} \nonumber \\
& = & 4 \sum_{p \leq Y} \frac{\log p}{p} + 4 \sum_{p \leq Y} \frac{\log p}{p} \cdot \frac{\epsilon_{p}^{2}}{(p/2)^{2} - \epsilon_{p}^{2}} \nonumber \\
& \geq & 2\log x -40\log\log x + 16\sum_{p \leq Y} \frac{\log p}{p} \frac{\epsilon_{p}^{2}}{p^{2}} + O(1), \nonumber
\end{eqnarray}
since $\sum_{p \leq Y} (\log p)/p = \log Y + O(1) = (1/2)\log x - 10\log\log x + O(1)$. But the left hand side cannot exceed $2(\log x + 1)$, because in that case we would have either
$$ \sum_{p \leq Y} \frac{\log p}{\nu_{\cal A}(p)} \geq \log x + 1 \;\;\;\;\; \textrm{ or } \;\;\;\;\; \sum_{p \leq Y} \frac{\log p}{\nu_{\cal B}(p)} \geq \log x + 1, $$
and then by the larger sieve (that is Lemma \ref{lem:Gallagherls}, with 
${\cal P}_0$ chosen as $\{p \leq Y\}$) we would have either $\mathcal{A}(x) \ll Y$ or $\mathcal{B}(x) \ll Y$, which we know to be false. Therefore we must have the following bound:
$$ \sum_{p \leq Y} \frac{\log p}{p} \frac{\epsilon_{p}^{2}}{p^{2}} \ll \log\log x. $$

We shall use the bound that we just derived to estimate the denominator in a final application of the large sieve. Firstly we note that, by the Cauchy--Schwarz inequality and standard results on the distribution of prime numbers,
$$ \sum_{\log x \leq p \leq Y} \frac{1}{p} \frac{|\epsilon_{p}|}{p} \leq \sqrt{\sum_{\log x \leq p \leq Y} \frac{1}{p \log p} \cdot \sum_{\log x \leq p \leq Y} \frac{\log p}{p} \frac{\epsilon_{p}^{2}}{p^{2}}} \ll \sqrt{\frac{1}{\log\log x} \cdot \log\log x} \ll 1. $$
In particular, this implies that $\sum_{\log x \leq p \leq Y; |\epsilon_{p}| \geq p/4} 1/p \ll 1$, and since we have $\sum_{Y < p \leq \sqrt{x}} 1/p \ll 1$ we clearly also have
$$ \sum_{\log x \leq p \leq \sqrt{x}} \frac{1}{p} \frac{|\epsilon_{p}|}{p} \ll 1 \;\;\;\;\; \text{and} \;\;\;\;\; \sum_{\log x \leq p \leq \sqrt{x}; |\epsilon_{p}| \geq p/4} \frac{1}{p} \ll 1 . $$
Now if we define a multiplicative function $f(n)$, supported on squarefree integers, by setting
$$f(p) := \left\{ 
\begin{array}{ll}
\frac{\nu_{\cal A}(p)}{p-\nu_{\cal A}(p)} 
= \frac{p/2+\epsilon_{p}}{p/2-\epsilon_{p}} 
& \textrm{if } |\epsilon_{p}| \leq p/4   \\
     0 & \textrm{otherwise}
\end{array} \right. , $$
then the large sieve (that is Lemma \ref{lem:Montgomery}, with the choice $Q = \sqrt{x}$) implies
$$ \mathcal{B}(x) \leq \frac{2x}{\sum_{n \leq \sqrt{x}} f(n)}. $$

It remains to show that $\sum_{n \leq \sqrt{x}} f(n) \gg \sqrt{x}/\log\log x$, 
and to do this we shall use the following nice result, which is a slightly
simplified version of Theorem 2 of Hildebrand \cite{Hildebrand:1987}.
\begin{theorem}[Hildebrand, 1987]
Let $X \geq Z \geq 2$. Let $f$ be a multiplicative function supported on squarefree numbers, and satisfying $0 \leq f(p) \leq K$ for some constant $K \geq 1$, for all primes $p$. Then
\begin{eqnarray}
\frac{1}{X} \sum_{n \leq X} f(n) & \geq & \frac{e^{-\gamma(K-1)}}{\Gamma(K)} \prod_{p \leq X} \left(1-\frac{1}{p} \right) \left(1 + \frac{f(p)}{p} \right) \cdot \nonumber \\
&& \cdot \left( \sigma_{-}\left(e^{\sum_{Z \leq p \leq X} (1-f(p))^{+}/p}\right)(1+O((\frac{\log Z}{\log X})^{\alpha})) + O(e^{-((\log X)/(\log Z))^{\alpha}}) \right), \nonumber
\end{eqnarray}
where $\gamma \approx 0.577$ is Euler's constant; $(1-f(p))^{+}$ denotes the positive part of $(1-f(p))$; $\alpha > 0$ is an absolute constant; the implicit constant in the ``big Oh'' terms depends on $K$ only; and $\sigma_{-}(u)$ is a certain fixed function of $u \geq 1$ that satisfies $\sigma_{-}(u) \gg u^{-u}$.
\end{theorem}
We can apply the theorem to the function $f(n)$ that we defined above, with the choices $X= \sqrt{x}$, $Z = \log x$ and $K=3$. We find that
$$ \sum_{\log x \leq p \leq \sqrt{x}} \frac{(1-f(p))^{+}}{p} \ll \sum_{\log x \leq p \leq \sqrt{x}, \atop |\epsilon_{p}| > p/4} \frac{1}{p} + \sum_{\log x \leq p \leq \sqrt{x}} \frac{|\epsilon_{p}|}{p^{2}} \ll 1, $$
in view of our previous calculations, so that
$$ \sum_{n \leq \sqrt{x}} f(n) \gg \sqrt{x} \prod_{p \leq \sqrt{x}} \left(1-\frac{1}{p} \right) \left(1 + \frac{f(p)}{p} \right) \gg \sqrt{x} \prod_{p \leq \log x} \left(1-\frac{1}{p} \right) \gg \frac{\sqrt{x}}{\log\log x}, $$
as claimed.
\qed

\bibliographystyle{alpha}

\noindent Christian Elsholtz,
Institut f\"ur Mathematik A, Technische Universit\"at Graz,
Steyrergasse 30/II, A-8010 Graz, Austria\\  
email: elsholtz@math.tugraz.at\\
\ \\
Adam J.~Harper,
Department of Pure Mathematics and Mathematical Statistics,
Wilberforce Road,
Cambridge CB\textup{3} \textup{0}WA,
England \\
email: A.J.Harper@dpmms.cam.ac.uk \\
\ \\
{\em Current address:} Centre de recherches math\'{e}matiques, Universit\'{e} de Montr\'{e}al, Case postale 6128, Succursale Centre-ville, Montr\'{e}al, QC H3C 3J7, Canada \\
email: harperad@crm.umontreal.ca

\ \\
\ \\
2010 Mathematics Subject Classification:\\
11N25, 11N36, 11P70
\end{document}